\documentclass[11pt]{article}
\usepackage{geometry}                
\usepackage{enumerate}
\usepackage{amssymb}
\usepackage{amsmath}
\usepackage{color}
\usepackage{amsthm}
\usepackage[pdftex]{graphicx}

\newtheorem{thm}{Theorem}[section]
\newtheorem{lem}{Lemma}[section]
\newtheorem{cor}[thm]{Corollary}
\newtheorem{prop}[lem]{Proposition}
\newtheorem{rem}[lem]{Remark}

\newtheorem{defn}[lem]{Definition}
\def\dr{\frac{d}{dr}}
\def\Tr{\operatorname{Tr}}
\def\tr{\operatorname{tr}}
\def\ind{\operatorname{ind}}
\def\a{\frac{a}{2}}

\def\k{\frac{k}{2}}

\def\i{\sqrt{-1}}
\def\l{\lambda}
\def\a{\alpha}

\def\R{\mathbb{R}}

\def\hc{\hat{c}}
\def\k{\kappa}
\def\r{r_{x}}

\def\C{\mathbb{C}}
\def\Im{\mathrm{Im}}

\def\d{\tilde{d}_T}

\def\half{\frac{1}{2}}

\def\n{\frac{n}{2}}
\def\e{\mathcal{E}}
\def\de{\delta_{\epsilon}}
\def\den{\delta_{\epsilon}^{-1}}
\def\ep{\epsilon}
\def\ce{c_\epsilon}
\def\Rt{\tilde{R}_{k,T}}

\def\hce{\hat{c}_\epsilon}
\def\epsq{\epsilon^{\frac{1}{2}}}
\def\epq{\epsilon^{\frac{1}{4}}}
\def\epnq{\epsilon^{-\frac{1}{4}}}
\def\epnsq{\epsilon^{-\frac{1}{2}}}

\def\pat{\frac{\partial}{\partial t}}

\def\t{\Theta}

\def\Vol{\operatorname{Vol}}
\def\Dom{\mathrm{Dom}}

\begin{document}
\def\o{\omega}
\def\O{\Omega}
\def\l{\lambda}
\title{Witten Deformation on Non-compact Manifolds\\
	-- {\Large Heat Kernel Expansion and Local Index Theorem}}
\author{Xianzhe Dai\thanks{Department of Mathematics, UCSB, Santa Barbara CA 93106, dai@math.ucsb.edu. Partially supported by the Simons Foundation}
	\and Junrong Yan\footnote{Department of Mathematics,  UCSB, Santa Barbara CA 93106, j\_yan@math.ucsb.edu}
}

\maketitle
\begin{abstract}
Asymptotic expansions of heat kernels and heat traces of Schr\"odinger operators on non-compact spaces are rarely explored, and even for cases as simple as $\mathbb{C}^n$ with (quasi-homogeneous) polynomials potentials, it's already very complicated. Motivated by path integral formulation of the heat kernel, we introduced a parabolic distance, which also appeared in Li-Yau's famous work on parabolic Harnack estimate. With the help of the parabolic distance, we derive a pointwise asymptotic expansion of the heat kernel for the Witten Laplacian with strong remainder estimate. When the deformation parameter of Witten deformation and time parameter are coupled, we derive an asymptotic expansion of trace of heat kernel for small-time $t$, and obtain a local index theorem. This is the second of  our papers in understanding Landau-Ginzburg B-models on nontrivial spaces, and in subsequent work, we will develop the Ray-Singer torsion for Witten deformation in the non-compact setting.
\end{abstract}
\section{Introduction}
\subsection{Overview}

Witten deformation is a deformation of the de Rham complex introduced in an extremely influential paper \cite{witten1982supersymmetry}. Witten deformation on closed manifolds has found many beautiful applications, such as the analytic proof of Morse inequalities,  the development of Floer homology theory, and Bismut-Zhang's proof \cite{bismutzhang1992cm} of Cheeger-M\"uller theorem (also known as the Ray-Singer conjecture).

The mathematical study of Landau-Ginzburg models has highlighted the question of understanding the Witten deformation on non-compact manifolds. In \cite{DY2020cohomology} we studied some of the fundamental questions in this regard, focusing on the relationships between the various cohomology theories involved. In this paper, we continue this study by looking into the heat kernel and index theoretic aspect of the Witten deformation on non-compact manifolds. In particular, one of our main results is a local index theorem for the Witten deformation on non-compact manifolds. For the very special case of Euclidean space $\mathbb C^n$ with a quasi-homogeneous polynomial, as is typical in Landau-Ginzburg models, the corresponding index theroem from our local index theorem reduces to the equality of the index with the Milnor number of the quasi-homogeneous polynomial, a result stated in \cite{fanfang2016torsion}. Local index theorems, besides their obvious interests, are important steps towards developing the theory of Ray-Singer analytic torsion and their related applications.

Recall that the Witten deformation deforms the de Rham complex $(\Omega^*(M), d)$  by the new differential 
$$d_{Tf}=d+ Tdf\wedge$$
where $f$ is a smooth function and $T$ is the deformation parameter. The spaces we focus on here are complete non-compact Riemannian manifolds $(M, g)$ with bounded geometry. The key to local index theory is the study of the heat kernel of the Witten Laplacian $\Box_{Tf}=(d_{Tf}+d^*_{Tf})^2$, in particular, its asymptotic expansion. In previous work \cite{fan2011schr}, \cite{fanfang2016torsion}, \cite{DY2020cohomology}, tameness conditions are imposed on the potential function $f$ in order for the Witten Laplacian to have discrete spectrums; here we introduce further tameness conditions which guarantee that the heat kernel of the Witten Laplacian is of trace class. In fact, we prove a weak Weyl Law for the eigenvalues of the Witten Laplacian. It is interesting to note that our tameness condition here is closely related to the semi-classical Weyl Law for Schr\"odinger operators in Euclidean space and are satisfied for the examples coming from Landau-Ginzburg models.

Developing the asymptotic expansion for the heat kernel of the Witten Laplacian presents further challenges in the non-compact setting, as we need a more refined remainder estimate so that the local index theorem can actually be integrated to an index theorem.
For the case of $\mathbb C^n$ with a quasi-homogeneous polynomial $f$, this is dealt with in \cite{fanfang2016torsion} by brute force, which does not give a needed strong remainder estimate nor generalize to the more general situation. To overcome the difficulty, we introduce a parabolic (meta-)distance which also appeared previously in the famous work of Li-Yau on Harnack inequality \cite{liyau1986}. This parabolic distance is also intimately related to the Agmon distance which plays a crucial role in our previous work \cite{DY2020cohomology}. The connection will then be exploited to establish the remainder estimate needed for the local index theorem. 

We would like to point out that the asymptotic expansion for the heat kernel of a Schr\"odinger operator as well as its trace on noncompact space is rarely explored. In the few cases it is studied (c. f. \cite{fanfang2016torsion} and \cite{fucci2018asymptotic}), it deals with the situation  when $M=\R^n$, $f$ is a polynomial. And even then, it is already very complicated. On the other hand,  though physicists are able to write a semi-classical asymptotic expansion of heat kernel via path integral, it is not sufficient for the asymptotic expansion of trace, as integration over noncompact space is involved. We found that the parabolic distance provides a much simpler and satisfying approach to the problem.


Another novel idea is in proving the local index theorem for the Witten deformation. Here  a modified Getzler rescaling involving the rescaling on the deformation parameter $T$ is used. Thus our result in \cite{DY2020cohomology} that the dimension of the $L^2$ cohomology of the Witten deformation is independent of $T$ (for $T>0$ or sufficiently large depending on the tameness conditions) plays a crucial role here.

In a separate paper we extend our treatment to Dirac/Callias type operators.
\newline

{\em Acknowledgment: We have benefitted from the preprint \cite{fanfang2016torsion} which has provided motivation and inspiration for us.  Thanks are also due to Guangbo Xu for interesting discussions a few years back when we started to look at these questions.
}

\subsection{Notations and Assumptions}
In this paper, we assume that all of our (Riemannian) manifolds have bounded geometry. Namely, 
\begin{defn}
Let $(M,g)$ be a complete Riemannian manifold with metric
$g$. $(M,g)$ is said to have bounded geometry, if the following
conditions hold:
\begin{enumerate}
\item The injectivity radius $\tau$ of $(M,g)$ is positive.
\item The curvature, as well as all its derivatives, is bounded, $|\nabla^m R|\le F_{m}$. Here $\nabla^m R$ is the $m$-th
covariant derivative of the curvature tensor and $C_m$ is a constant
only depending on $m$.
\end{enumerate}
\end{defn}

%

We now introduce the tameness conditions we need in treating the local index theorem. Various notions of tameness have been introduced, for example, the strong tameness condition in \cite{fan2011schr} and the weaker well tameness in \cite{DY2020cohomology}. Here we need stronger tameness conditions.

\begin{defn}
Let $(M,g)$ be a complete Riemannian manifold and $\k\in [0,1)$. We say $(M,g,f)$ is $\kappa$-regular tame if 
\begin{enumerate}
\item $\limsup_{p\to\infty} \frac{|\nabla^m f|}{|\nabla f|^{(m-1)\k+1}}<\infty,$ for any $m\geq 1$;
\item $\lim_{p\to\infty} |\nabla f|=\infty.$
\end{enumerate}
Here $\nabla f$ denotes the  gradient of $f$, $\nabla^m f= \nabla^{m-1} \nabla f, m\geq 1$, and $p\to\infty$ means that $d(p,p_0)\to \infty$ for the distance function from a fixed base point $p_0$.
\end{defn}

In typical examples from Landau-Ginzburg models, $M=\mathbb C^n$ with the Euclidean metric and a nondegenerate quasi-homogeneous polynomial $f$. Then $(\mathbb C^n, f)$ is $\kappa$-regular tame for some $\kappa < 1$, see the discussion in the last section. For our purpose, we reformulate one of the consequences of $\kappa$-regular tameness. Indeed, an inductive argument yields that, if $(M,g,f)$ is $\kappa$-regular tame, then for $V=|\nabla f|^2$, we have, for all $k\in \mathbb N$, 
\[\limsup_{p\to\infty}\frac{|\nabla^kV|}{|V|^{(k\k+2)/2}} < \infty.\]
\begin{rem}
In \cite{rozenbljum1974asymptotics}, in order to prove Weyl's law for Schr\"odinger operator on $\R^n$, Rozenbljum imposed similar $\k$-tameness conditions (see (0.6) in \cite{rozenbljum1974asymptotics}). Later we will show that with $\k$-tameness condition, one can prove a weaker version of Weyl's law.
\end{rem}

The  $\kappa$-regular tame condition ensures that, among other things, the Witten Laplacian has discrete spectrums, Cf. \cite{fan2011schr, DY2020cohomology}. As far as we know, examples coming from Landau-Ginzburg models satisfy the regular tame condition, see the discussion in the last section for a large class of such examples. 

Our next condition ensures that we have a good local index theory, as we will see later.

\begin{defn}
Fix $\a\geq n/2.$ A triple $(M,g,f)$ is called $\a$-polynomial tame, if $(M,g,f)$ is $\kappa$-regular tame for some $\k\in(0,1)$, and in addition, there is some constant $C$, such that for all $\lambda \geq 0$, 
\[\int_{\{p\in M:|\nabla f|^2(p)\leq \lambda\}}(\lambda-|\nabla f|^2)^{n/2}dvol_M\leq C\lambda^\a. \] 
\end{defn}

Again we will see that typical examples coming from Landau-Ginzburg models are polynomial tame, see the discussion in the last section.

\begin{rem}
This condition should be interpreted in terms of the (semiclassical) Weyl's law for Schr\"odinger operators in Euclidean space ( Cf. \cite{rozenbljum1974asymptotics,tachizawa1992eigenvalue} ), which would guarantee the polynomial growth of the eigenvalues. However, though expected,  we could not find in the literature such Weyl's law on manifolds. To focus our discussion on the asymptotic expansion of heat kernel and local index theorem, we only prove a weaker version of Weyl's law here under our assumption. In a separate paper we will prove such Weyl's law by using similar techniques in \cite{tachizawa1992eigenvalue} and also give a treatment using heat kernel following the techniques developped here in Sections 3 and 4.
\end{rem}

\subsection{Main Result}
As we will see in the next section, the $\a$-polynomial tame condition garantees that $\exp(-t\Box_{Tf})$ is of  trace class. 
Our first contribution is the pointwise asymptotic expansion of the heat kernel of the Witten Laplacian with strong remainder estimate. 

Let $(M, g, f)$ be $\a$-polynomial tame, and $K_{Tf}(t, x, y)$ denote the heat kernel of the Witten Laplacian $\Box_{Tf}$. Denote by $h_T(x, y)$ the average of the potential function $T^2 |\nabla f|^2$ on the geodesic segment from $x$ to $y$, Cf. \eqref{hT}.

\begin{thm} \label{main-heatasym} The heat kernel $K_{Tf}$ has the following complete pointwise asymptotic expansion. For any $x, y\in M$ such that $d(x, y) \leq \half \tau$, 
	\[K_{Tf}(t,x,y) \sim \frac{1}{(4\pi t)^{\n}}\exp(-d^2(x,y)/4t)\exp(-t\, h_T(x,y)) 
	\sum_{j=0}^{\infty}t^j\t_{T, j}(x,y),  \]
	as $t\rightarrow 0$. 
	Each $\t_{T,j}$ is a polynomial of $T$:
	\[\t_{T,j}(x,y)=\sum_{l=0}^{[\frac{j}{3}]+j}T^l\t_{l,j}(x,y),\]
	and, when restricted to the diagonal of $M\times M$, $\t_{l,j}(y,y)$ can be written as an algebraic combination of the curvature of the metric $g$, the function $f$, as well as their derivatives, at $y$; in addition, $\t_{T,0}(y,y)=\operatorname{Id}.$
	Moreover, we have the following remainder estimate. For any $k$ sufficiently large and any
	$a\in (0, 1)$,  
	\begin{align*}
	\left| K_{Tf}(t,x,y) - \frac{1}{(4\pi t)^{\n}}\exp(-d^2(x,y)/4t)\exp(-t\, h_T(x,y)) 
	\sum_{j=0}^{k}t^j\t_{T, j}(x,y) \right| \\
	\leq Ct^{\frac{1}{3}(1-\k)k -\frac{\k +2}{3} -\n +1}T^{\frac{-2k+4}{3}}\exp(-a \d(t,x,y)) , \hspace{1.5in}
	\end{align*} 
	for $t\in(0,1]$ and $T\in(0,t^{-\half}]$. 
\end{thm}

Here $\d(t,x,y)$ is the parabolic distance alluded at the beginning of the introduction, see \eqref{d0} for the precise definition. By relating it to the Agmon distance we are able to obtain an effective bound on $\d(t,x,y)$, which, when combined with the theorem above, yields the following corollary.

\begin{cor}	\label{cor-dheatasym} 
For $T=t^{-\half}$, and any $k$ sufficiently large, any	$a\in (0, 1)$,  
		\begin{align*}
	\left| K_{t^{-\half}f}(t,x,x)- \frac{1}{(4\pi t)^{\n}}\exp(- |\nabla f|^2(x)) 
	\sum_{j=0}^{\infty}\sum_{l=0}^{[\frac{j}{3}]+j}t^{j-\frac{l}{2}}\t_{l,j}(x,x) \right|  \\
	\leq Ct^{\frac{1}{3}(2-\k)k -\frac{\k +1}{3} -\n}\exp(-a\bar{\beta} |\nabla f|^{1-\k}(x) ) , \hspace{.8in}
	\end{align*} 
	for $t\in(0,1]$, where $\bar{\beta}>0$ is a constant depending only on the bounds in the tameness condition. In particular, we have the following small time asymptotic expansion of the heat trace:
	\[ \operatorname{Tr}_s\left(\exp (-t\Box_{t^{-\frac{1}{2}}f})\right) \sim \frac{1}{(4\pi t)^{\n}} 
	\sum_{j=0}^{\infty} \sum_{l=0}^{[\frac{j}{3}]+j}t^{j-\frac{l}{2}} \int_M \exp(-|\nabla f|^2(x)) \operatorname{tr}_s(\t_{l, j}(x,x)) dx,  \]
	as $t\rightarrow 0$, with the remainder estimate
	\begin{align*}
	\left| \operatorname{Tr}_s\left(\exp (-t\Box_{t^{-\frac{1}{2}}f})\right)- \frac{1}{(4\pi t)^{\n}} 
	\sum_{j=0}^{k} \sum_{l=0}^{[\frac{j}{3}]+j}t^{j-\frac{l}{2}} \int_M \exp(-|\nabla f|^2(x)) \operatorname{tr}_s(\t_{l, j}(x,x)) dx\right|  \\
	\leq Ct^{\frac{1}{3}(2-\k)k -\frac{\k +1}{3} -\n }. \hspace{1.5in}
	\end{align*} 
Here $\Tr_s$ and $\tr_s$ denote the global supertrace and pointwise supertrace respectively.
\end{cor}

On the other hand, by Theorem 1.3 in \cite{DY2020cohomology} and the $\a$-polynomial tame condition, the index of the Witten Laplacian
\[\chi(M,d_{Tf})=\sum_{i=0}^n (-1)^ib_i(T),\ b_i(T)=dim(H^i_{(2)}(M,d_{Tf}))\]
 is independent of $T>0$. In fact we have that $\Tr_s(exp(-t\Box_{Tf}))$ is independent of $t$ and $T$ and $$\chi(M,d_{Tf})=\ind(\Box_{Tf})=\Tr_s(-t\exp(\Box_{t^{-\half}f}))= \int_M \operatorname{tr}_s(K_{Tf}(t,x,x)) dx.$$
Now apply our new rescaling technique, one has

\begin{thm}[Local index theorem and index formula for $\Box_{Tf}$]\label{main1}
For any $x\in M$, we have 
\[ \lim_{t\to 0}\operatorname{tr}_s(K_{t^{-\frac{1}{2}}f}(t,x,x))= \frac{(-1)^{[\frac{n+1}{2}]}}{\pi^\n}\exp(-|\nabla f(x)|^2)\int^B\exp(-\frac{\widetilde{R}(x)}{2}-\widetilde{\nabla}^2 f(x)).
\]
In particular, for any $T>0$,
\begin{equation}\label{index}
\ind(\Box_{Tf})=\frac{(-1)^{[\frac{n+1}{2}]}}{\pi^\n}\int_{M}\exp(-|\nabla f|^2)\int^B\exp(-\frac{\tilde{R}}{2}-\tilde{\nabla}^2 f).
\end{equation}
\end{thm}

Here  $\int^B$ denotes the Berezin integral, which will be introduced in a moment, and  $\tilde{R},\tilde{\nabla}^2f\in \O^*(M)\hat\otimes \O^*(M)$ are defined as 
\[\tilde{R}=-\sum_{i<j,k<l}R_{ijkl}e^ie^j\hat{e}^k\hat{e}^l, \ \ \ \  \tilde{\nabla}^2f=\nabla^2_{e_i,e_j} fe^i\hat{e}^j\] 
for some orthonormal frame $\{e_i\}$ in $TM$ and its dual frame $\{e^i\}$ in $T^*M$. We have used $\{\hat{e}_i\}$ to denote the same orthonormal frame in the second copy of $T^*M$. For any $\o\in \O^*(M)\hat\otimes \O^*(TM)$, $I=\{i_1,...,i_k\}\subset {1,2,...,n}$, we write $\o$ as
\[\o=\sum_{I}w_I\hat{e}^I,\]
where $\hat{e}^I=\hat{e}^{i_1}\wedge...\wedge \hat{e}^{i_k}.$
Then the Berezin integral is defined as
\begin{equation*} \int^B:\, \O^*(M)\hat\otimes \O^*(M)\mapsto \O^*(M), \ \ \ \ \ \ 
\int^B \o=\o_{1,2,...,n}.
\end{equation*}

\begin{rem} 
\ \\	
\begin{enumerate}
\item Here the index density is computed by coupling $tT^2=1$.  Our arguments still work if we set $tT^2$ to be any positive constant $T_0$. As $T_0\to \infty$, the integral of index density localizes  at critical points of $f$. On the other hand, when $T_0\to 0^+$,  the index of $\Box_{Tf}$ should depend on "the topology away from infinity" and the behavior of $f$ near the infinity. This will be discussed in more detail in a separate paper where we extend our treatment to Dirac/Callias type operators.

\item When $M$ is compact, (\ref{index}) is a special case of a formula in chapter 3 of \cite{zhang2001lectures}.

\item Notice that $\int^B\exp(-\nabla^2 f)=(-1)^{[\n]}\det(-\nabla^2 f)$. Thus, when $M=\R^n$, (\ref{index}) reduces to 
\[\chi(\R^n,d_f)=\frac{(-1)^n}{\pi^\n}\int_{\R^n}\exp(-|\nabla f|^2)\det(-\nabla^2f)dvol.\]
In particular, when $M=\C^n$, $f$ is a holomorphic function such that its real part $\Re f$ is polynomial tame, we have
\begin{align*}\chi(\C^n,d_f)&=\frac{1}{\pi^n}\int_{\C^n}\exp(-|\nabla \Re f|^2)det(-\nabla^2\Re f)dvol\\
&=\frac{(-1)^n}{\pi^n}\int_{\C^n}\exp(-|\partial f|^2)|det(-\partial^2f)|^2dvol \end{align*}
is given by the Milnor number of $f$. This is a generalization of a result in \cite{fanfang2016torsion}, see the last section for more discussion. 
\end{enumerate}
\end{rem}

\section{Weak Weyl Law} \label{wwl}

In this section we will show that the polynomial tame condition implies that $\exp(-t\Box_{Tf})$ is of  trace class. This is achieved by proving a weak Weyl law which shows that the eigenvalues of the Witten Laplacian grows polynomially. The Agmon estimate developped in \cite{DY2020cohomology} plays a crucial role here.

\subsection{Review of Hodge Theory for Witten Laplacian} \label{rohtfwl}

For any {$T> 0$}, let 
$$d_{Tf}:=d+Tdf\wedge:\Omega^*(M)\mapsto \Omega^{*+1}(M)$$
 be the so-called Witten deformation of de Rham operator $d$. As usual,  the metric $g$ induces a canonical metric (still denote it by $g$) on $\Lambda^*(M)$, which then defines an inner product $(\cdot,\cdot)_{L^2}$ on $\Omega^*_c(M)$:
\[(\phi,\psi)_{L^2}=\int_M (\phi, \psi)_gdvol,\phi,\psi\in \Omega^*_c(M). \]

Let $L^2\Lambda^*(M)$ be the completion of $\Omega_c^*(M)$ with respect to  $\|\cdot\|_{L^2}$, and $L^2(M):=L^2\Lambda^0(M).$

Then $d_{Tf}$ is an unbounded operator on $L^2\Lambda^*(M)$ with domain $\Omega^*_c(M)$. Also, it has a formal adjoint operator $\delta_{Tf}$, with $\Dom(\delta_{Tf})=\Omega^*_c(M),$ such that
\[(d_{Tf}\phi,\psi)_{L^2}=(\phi,\delta_{Tf}\psi)_{L^2},\phi,\psi\in\Omega^*_c(M).\]

Set $\Delta_{H,Tf}=(d_{Tf}+\delta_{Tf})^2,$ and we denote the Friedrichs extension of $\Delta_{H,Tf}$ by $\Box_{Tf}$. If $(M, g)$ is complete then $\Delta_{H,Tf}$ is essentially self-adjoint (and hence  $\Box_{Tf}$ is the unique self-adjoint extension). In \cite{DY2020cohomology}
We proved that when $(M, g, f)$ is tame,
\begin{equation}\label{dec}L^2\Lambda^*(M)=\ker\Box_{Tf}\oplus\Im \bar{d}_{Tf}\oplus\Im \bar\delta_{Tf},\end{equation}
where $\bar{d}_{Tf}$ and $\bar\delta_{Tf}$ are the graph extensions of $d_{Tf}$ and $\delta_{Tf}$ respectively.

Setting $\Omega_{(2)}^*(M, Tf):=\Dom(\bar{d}_{Tf})\cap\Omega^*(M),$  we have a chain complex
\[\cdots \xrightarrow{d_{Tf}}\Omega_{(2)}^*(M, Tf)\xrightarrow{d_{Tf}}\Omega_{(2)}^{*+1}(M, Tf)\xrightarrow{d_{Tf}}\cdots.\]
Let $H^*_{(2)}(M,d_{Tf})$ denote the cohomology of this complex. In \cite{DY2020cohomology}, we have shown that $H^*_{(2)}(M,d_{Tf})\cong\ker\Box_{Tf}$, provided $(M, g, f)$ is well tame and $T$ is large enough. Note that the notion of well tame \cite{DY2020cohomology} is strictly weaker than that of regular tame.

Finally, we note the following well known

\begin{prop}\label{estihdg}
	Denote $L_f=\nabla^2_{e_i,e_j}f[e^i\wedge,\iota_{e_j}]$ locally,
	where $\{e_i\}$ is a local frame on $TM$ and $\{e^i\}$ is the dual frame on $T^*M.$
	Then the Witten Laplacian  
	$\Delta_{H,Tf}$ has the
	following expression:
	\begin{equation}
	\Delta_{H,Tf}=\Delta-TL_f+T^2|\nabla f|^2.
	\end{equation}
	Here $\Delta$ denotes the Hodge Laplacian.
\end{prop}	

\subsection{Weak Weyl Law for Witten Laplacian} \label{wwlfwl}

Let  $(M,g,f)$ be $\a$-polynomial tame defined in the previous section. Then,   $(M,g,f)$ is regular tame and there is some constant $C$, such that for all $\lambda \geq 0$, 
\[\int_{\{p\in M:|\nabla f|^2(p)\leq \lambda\}}(\lambda-|\nabla f|^2)^{n/2}dvol_M\leq C\lambda^\a. \] 
This has the following immediate consequences.

\begin{lem}\label{Klambda}
	Let $K_{\lambda}:=\{p\in M:|\nabla f|^2(p)<\lambda\},$ then 
	\[\Vol(K_\lambda)\leq C\lambda^{\a-\n}.\]
	Furthermore, for any $k\geq 0$, there is a constant $C_k$ depending only on $k$ and the tameness condition such that
	\[\int_M\exp(-|\nabla f|^2)|\nabla f|^k dvol \leq C_k . \]
\end{lem}
\begin{proof} We have
	\begin{align*}
	\lambda^{\n}\Vol(K_\lambda)\leq \int_{K_{2\lambda }}(2\lambda-|\nabla f|^2)^{\n}\leq C\lambda^\a.
	\end{align*}
	To prove the second estimate, we notice that 
	\begin{align*} \int_M\exp(-|\nabla f|^2)|\nabla f|^k dvol & = \sum_{l=0}^{\infty } \int_{K_{l+1} - K_l}\exp(-|\nabla f|^2)|\nabla f|^{\frac{k}{2}} dvol \\
	& \leq \sum_{l=0}^{\infty } e^{-l}(l+1)^{\frac{k}{2}} \Vol(K_{l+1} - K_l)  \\
	& \leq C\sum_{l=0}^{\infty } e^{-l}(l+1)^{{\frac{k}{2}}+\a -\n}=C_k < \infty, 
	\end{align*}
	as desired.
\end{proof}

Note that, in particular,  if $\a=n/2$, then $M$ must have finite volume.

We now turn our attention to the growth of eigenvalues of the Witten Laplacian. First, by  refining the  argument of Theorem 1.1 in \cite{DY2020cohomology}, we have the following exponential decay estimate for eigenforms.
\begin{prop}\label{Agmon}
	Let $(M,g,f)$ be strongly tame, and $\o\in\Dom(\Box_{f})$ be an eigenform of $\Box_{f}$ with eigenvalue $\lambda$. Then
	\[|\o(p)|\leq C\exp(-a\rho_{\lambda}(p))\|\omega\|_{L^2},\]
	for any $a\in(0,1)$. Here $\rho_\lambda$ is the Agmon distance induced by Agmon metric $g_\lambda:=(|\nabla f|^2-\lambda)_+\, g$, with $(|\nabla f|^2-\lambda)_+$ denoting the nonnegative part, and $C$ is a constant independent of $\lambda.$
\end{prop}

With the help of Proposition \ref{Agmon}, we now deduce a weak version of Weyl's law:

\begin{prop}\label{weyl}
	If $(M,g,f)$ is $\a$-polynomial tame, then the spectrum of $\Box_{f}$ has polynomial growth. More precisely, there exist constants $\delta>0$ and  $C>0$, such that $\lambda_k(\Box_{f})\geq Ck^{\delta}$, where $\lambda_k(\Box_{f})$ denotes the $k$-th eigenvalue of $\Box_{f}$ (counted with multiplicity). Consequently, $\exp(-t\Box_{Tf})$ is of trace class for all $T>0,t>0.$
\end{prop}
\begin{proof}
	Let $E(\lambda)$ be the number of eigenvalues not exceeding $\lambda$, and $u$ be an eigenform with eigenvalue $\lambda_0\leq \lambda$. We normalize $u$ so that $\|u\|_{L^2}=1$ .
	
	By Proposition \ref{Agmon}, 
	$$
	|u(p)|\leq C\exp(-a\rho_{\lambda}(p)).
	$$
	We claim that there exists $n_0>1$ independent of $\lambda \geq 1$ and $u$, such that  
	\begin{equation}\label{noncpt}\int_{M-K_{n_0\lambda}}|u|^2dvol <\frac{1}{2}.\end{equation}
	
	To prove the claim we first estimate the Agmon distance. Thus,  for $p\in K_{(k+1)\l}-K_{k\l}$, let $\gamma:[0,l]\mapsto M$ be a minimal curve in the Agmon metric  $g_\l$ connecting $p$ and $K_{k\l}$; moreover, $|\gamma'(s)|=1$ with respect to the  metric $g.$
	Then we may as well assume that $\gamma\subset K_{(k+1)\l}$; otherwise, we can find $l_0\in [0,l]$, such that $\gamma|_{[l_0,l]}\subset K_{(k+1)\l}$ and we can take $p=\gamma(l_0)$. Hence by the tameness condition, there exists $c>0$, s.t. 

	\[ \frac{d}{dt}(|\nabla f|^2\circ \gamma(t))\leq c|\nabla f|^{\k+2}\leq c((k+1)\lambda)^{\frac{\k+2}{2}}.\]
	
	It follows by integrating that  $l\geq\frac{|\nabla f|^2(p)-k\l}{((k+1)\lambda)^{\frac{\k+2}{2}}}.$
	In particular, if $L$ is the  $g_\l$-length of $\gamma$ such that that $|\nabla f|^2(p)=(k+1)\l,$ then for some $c'>0$
	
	\[L= \int_0^l(|\nabla f|^2-\l)^{\half}\circ\gamma(t)dt\geq 
	\frac{(k-1)^\half \l^{\frac{1-\k}{2}}}{(k+1)^{\frac{\k+2}{2}}}\geq \frac{c'\l^{\frac{1-\k}{2}}}{k^{\frac{\k+1}{2}}}.\]
	
	Hence, if $x\in K_{(k+1)\l}-K_{k\l}$, then (say $k\geq 3$)
	\[\rho_\l(x)\geq \sum_{i=2}^{k-1}\frac{c'\l^{\frac{1-\k}{2}}}{{i}^{\frac{\k+1}{2}}}\geq c''\l^{\frac{1-\k}{2}}k^{\frac{1-\k}{2}} \]
for some 
constant $c''>0$.
	
Therefore, 
	\begin{align*}
	\int_{M-K_{n_0\l}}|u|^2dvol&=\sum_{k=n_0}^\infty\int_{K_{(k+1)\l-K_{k\l}}}|u|^2dvol\\
	&\leq \sum_{k=n_0}^\infty\int_{K_{(k+1)\l-K_{k\l}}}Ce^{-a\rho_\l}dvol\\
	&\leq \sum_{k=n_0}^\infty Ce^{-ac''\l^{\frac{1-\k}{2}} k^{\frac{1-\k}{2}}} \Vol(K_{(k+1)\l})\\
	&\leq\sum_{k=n_0}^\infty C_1 e^{-ac''\l^{\frac{1-\k}{2}} k^{\frac{1-\k}{2}}} ((k+1)\l)^{\a-\n}\\
	&\leq\sum_{k=n_0}^\infty C_2 C'e^{-\frac{1}{2}ac''\l^{\frac{1-\k}{2}} k^{\frac{1-\k}{2}}} \leq\sum_{k=n_0}^\infty C_2C'e^{-\frac{1}{2}ac''k^{\frac{1-\k}{2}}}
	\end{align*}
for $\lambda \geq 1$. Here $C'=\max_{\eta>0} \eta^{\frac{2\a-n}{1-\k}}e^{-\frac{1}{2}ac''\eta}$.  Clearly there is some $n_0$ such that the last term in the inequality above is less than $1/2$, 	which finishes the proof of the claim.
	
	Let $N(\ep,\lambda)$ be the minimal number of elements in an $\ep$-dense subset of $K_{n_0\lambda}$. Then by the volume comparison, $N(\ep,\lambda)\leq C_1\frac{\Vol(K_{n_0\lambda})}{\ep^{n}}$. 
	We now follow the argument in the proof of Theorem 5.8 of \cite{lawson2016spin} to show that $E(\lambda)\leq N(\ep,\lambda)$ for suitable $\ep$.  Indeed,  if $E(\lambda)>N(\ep,\lambda)$, then there exists $u\in E(\lambda)$ with unit $L^2$ norm   which vanishes on an $\ep$-dense subset of $K_{n_0\lambda}$. By using the elliptic estimate as in \cite{lawson2016spin} one deduces
	\[\sup_{K_{n_0\lambda}}|u|\leq \ep C_k(1+\lambda^k) \]
	for any $2k>\n+1.$ But this is clearly impossible if we take $\ep^{-1}:=2C_k(1+\lambda^k)\Vol(K_{n_o\lambda})^{1/2}$, as 
	$$
	\int_{K_{n_0\l}}|u|^2dvol > 1/2.
	$$. 
	
	As a result, if we choose the minimal $k$, s.t. $2k>\n+1$, then by Lemma \ref{Klambda},
	\[E(\lambda)\leq N(\ep,\lambda)\leq C_1\frac{\Vol(K_{n_0\lambda})}{\ep^{n}}\leq C\lambda^{\n \a+\a+n}.\]
	The rest of the proposition follows.
\end{proof}
\begin{rem}
	The $\a$-polynomial tame condition is a technical one for the usual heat kernel approach to local index theorems. For example, on $\R$ consider $f=|x|\ln|x|$  outside $|x|\leq e$.  Let $\lambda_k$ be the $k$-th eigenvalue of $\Box_{f}$. Then by Weyl's law (Cf. \cite{tachizawa1992eigenvalue}), $\lambda_k\lesssim \sqrt{\ln(k)}$. For such slowly growing eigenvalue distributions, it is unreasonable to consider the limit $\lim_{t\to0}Tr_s(\exp(-t\Box_{f}))$.  On the other hand, this assumption is not essential if one is only interested in an index formula. This issue will be elaborated in a separate paper when we discuss the Dirac/Callias type operators.
\end{rem}
	Thus, assuming the $\a$-polynomial tame condition,  $\exp(-t\Box_{Tf})$ is of  trace class. 
	It follows that
	 \begin{equation} \label{stohk}
	 	h(t,T)=\Tr_s(\exp(-t\Box_{Tf}))
	 \end{equation}
	  is independent of $t.$  Moreover,  as $t\to \infty$, $h(t,T)\to \chi(M,d_{Tf}),$ where
	\[\chi(M,d_{Tf})=\sum_{i=0}^n (-1)^ib_i(T),b_i=\dim(H^i_{(2)}(M,d_{Tf})).\]
	
	Now by Theorem 1.3 in \cite{DY2020cohomology}, $h(t,T)$ is independent of $T>0$. As a result, $h(t,T)$ is independent of both $t>0$ and $T>0$. 

\section{Construction of Parametrix}\label{expansion} 

In this section, we extend the parametrix construction of the heat kernel to the Witten deformation. The case of Euclidean space is treated in \cite{fanfang2016torsion}.

Fix $x\in M$, and let $d(y,x)$ be the distance function. Let $\tau>0$ be the injectivity radius of $M.$ Then for $y\in B_{\tau}(x)$, define
\begin{equation} \label{ehk}
\e_0(t,x,y)=\frac{1}{(4\pi t)^{\n}}\exp(-d^2(x,y)/4t).
\end{equation}
 For simplicity, we denote $V_T=T^2|\nabla f|^2$ and $V=|\nabla f|^2$. Suppose $\gamma$ is the normal geodesic connecting $x$ and $y$, and $r_x(y)=d(x, y)$. Set \begin{equation} \label{hT} h_T(x,y)=\frac{1}{r_{x}(y)}\int_0^{\r(y)}V_T(\gamma(s))ds= T^2 h(x, y), \ \ \ h(x,y)=\frac{1}{r_{x}(y)}\int_0^{\r(y)}V(\gamma(s))ds .
 \end{equation}  We define
 \begin{equation} \label{cthk}
 \e_{1,T}(t,x,y)=\exp(-t\, h_T(x,y)).
 \end{equation} 

Then direct computation gives us the following formulas (the first two are well known). 
\begin{prop}\label{e0}
For $y\in B_\tau(x)$ in the  normal coordinates near $x$, we have
\[\nabla \e_0=-\frac{\e_0}{2t}r\nabla r,\ \ \ (\pat +\Delta)\e_0=\frac{\e_0}{4tG}\nabla_{r\nabla r}G.\]
\[\nabla_{r\nabla r}h_T(x,y)+h_T(x,y)-V_T(y)=0.\]
Here $G=\det(g_{ij})$ and derivatives are  taken with respect to $y.$
\end{prop}

Let $p_i:M\times M\mapsto M$ be the projection of i-th factor of $M\times M$ to $M, \ i=1,2$. We define the vector bundle $E\to M\times M$ to be
$E=(p_1)^*(\Lambda^*(M))\otimes (p_2)^*(\Lambda^*(M))$.
Let $s(t,x,y)=\sum_{i=0}^kt^i\t_i(x,y),$ where $\t_i(x,y)\in\Gamma(E)$. 
Since $y$ is within the injectivity radius of $x$, we use parallel transport along radius geodesics to identify $\Lambda^*_y(M)$ with $\Lambda^*_x(M)$. In this way, $\t_i(x,y)\in\Gamma(E)$ is identified with an endomorphismm  of $\Lambda^*(M)$ using the metric.  Again by a straightforward computation and using Proposition \ref{e0}, we have
\begin{prop}
\begin{flalign}
\begin{split}\label{asym}
(&\pat+\Box_{Tf})(\e_0\e_{1,T}s)\\
&=\e_0\e_{1,T}\Big\{\sum_{j=-1}^{k-1}[(j+1+\frac{1}{4G}\nabla_{r\nabla r}G)\t_{j+1} + \nabla_{r\nabla r}\t_{j+1} + \Delta \t_j-TL_f \t_j]t^j\\
& + [\Delta \t_k-TL_f \t_k]t^k
+\sum_{j=1}^{k+1}[-\Delta h_T\t_{j-1}+2\nabla_{\nabla h_T} \t_{j-1}]t^j\\
&+\sum_{j=2}^{k+2}[-|\nabla h_T|^2\t_{j-2}]t^j\Big\},
\end{split}
\end{flalign}
where 
the derivatives are taken with respect to $y.$
\end{prop}


Now we can follow the standard procedure  to find suitable $\t_j=\t_{T,j}$ with $\t_{T,0}(x,x)=\operatorname{Id}$, $j=0,1,...,k,$ such that 
\begin{equation} \label{parametrix}
(\pat+\Box_{Tf})(\e_0\e_1s)=t^kR_{k,T}(t,x,y),
\end{equation}
 where $R_{k,T}(t,x,y)$ is $C^0$ in $t\in[0,\infty)$.
This amounts  to solving ODEs inductively.

For $j=-1$, we have
$\dr(G^{1/4 }\t_{T,0})=0$. Together with the initial condition  $\t_{T,0}(x,x)=\operatorname{Id}$, one has 
$\t_{T,0}=G^{-1/4}\operatorname{Id}$.

For $j=0$, we have
$\dr(rG^{1/4 }\t_{T,1})=G^{1/4}(TL_f-\Delta) \t_{T,0}$; hence we can solve $\t_1$ explicitly in terms of $\t_{T,0}$, by integrating along the geodesic.

Similarly, for $1\leq j\leq k-1$, $\t_{T,j+1}$ can be solved recursively from the equation 
\begin{flalign*}
\dr(r^{j+1}G^{1/4}\t_{T,j+1})&=-r^{j}G^{1/4}(\Delta\t_{T,j}-TL_f \t_{T,j}-\Delta h_T\t_{T,j-1}\\
&+2\nabla_{\nabla h_T} \t_{T,j-1}-|\nabla h_T|^2\t_{T,j-2}).
\end{flalign*}
With these choices for  $\t_{T,j}$'s, we obtain (\ref{parametrix}), where
\begin{flalign} \label{asyrem}
R_{k, T}&=\e_0\e_1\Big\{[\Delta \t_{T,k}-TL_f \t_{T,k}-\Delta h_T\t_{T,k-1}+2\nabla_{\nabla h_T} \t_{T,k-1}-|\nabla h_T|^2\t_{T,k-2}]  \nonumber \\ 
&+[-\Delta h_T\t_{T,k}+2\nabla_{\nabla h_T} \t_{T,k}-|\nabla h_T|^2\t_{T,k-1}]t+[-|\nabla h_T|^2\t_{T,k}]t^{2}\Big\}
\end{flalign}

The following proposition follows from the above construction via an argument of induction, using the $\kappa$-regular tame condition.
\begin{prop}\label{asyexp}
Each $\t_{T,j}$ can be written as a polynomial of $T$:
\[\t_{T,j}(x,y)=\sum_{l=0}^{[\frac{j}{3}]+j}T^l\t_{l,j}(x,y),\]
where $\t_{l,j}$ is independent of $T$, $[a]$ denotes the integral part of a real number $a$. Moreover 
\[ |\t_{T,j}(x,y)|\leq C(\bar{V}_\gamma)^{\k' j}T^{[\frac{j}{3}]+j}, \]
where $\k'=\frac{\k+2}{3}$, $\bar{V}_{\gamma}=\sup_{p\in\gamma} |V(p)|$, $\gamma$ is the shortest geodesic connecting $x$ and $y.$
When restricted to the diagonal of $M\times M$, $\t_{T,j}(y,y)$ can be written as an algebraic combination the curvature of the metric $g$, the function $f$, as well as their derivatives, at $y$; in addition, $\t_{T,0}(y,y)=\operatorname{Id}.$
\end{prop}

Let $\eta\in C_c^\infty(\R)$ be a bump function, such that the support of $\eta$ is contained in $[-1,1],$ and $\eta|_{[-\half,\half]}\equiv1.$ 
Let $\phi\in C^\infty(M\times M)$ be defined as
\begin{equation} \label{cutoff}
\phi(x,y)=\eta(d^2(x,y)/\tau^2). 
\end{equation}

\begin{prop}\label{asym}
Set \[K_{Tf}^k(t,x,y)=\phi(x,y)\e_0(t,x,y)\e_{1,T}(t,x,y)\sum_{j=0}^kt^j\t_{T,j}(x,y),\] 
then 
\begin{flalign*}
(\pat+\Box_{Tf})K^k_{Tf}(t,x,y)&=t^k  \phi(x,y)R_{k,T}(t,x,y)+\Delta\phi(x,y)K^k_{Tf}(t,x,y)  \\
&  -2 (\nabla\phi(x,y),\nabla K_{Tf}^k(t,x,y)) , 
\end{flalign*}
where $R_{k,T}$ is given by (\ref{asyrem}).
\end{prop}

The following lemma provides the estimate saying that $K_{Tf}^k(t,x,y)$ is a suitable parametrix for the heat kernel of the Witten Laplacian. The proof uses Lemma \ref{vexpv} which will be shown in the next section when we introduce the necessary notions.
\begin{lem}\label{estRk}
	Assume $t\in(0,1]$. Let 
	\begin{flalign*}
	\Rt=t^k \phi(x,y)R_{k,T}(t,x,y)+\Delta\phi(x,y)K^k_{Tf}(t,x,y)-2 (\nabla\phi(x,y),\nabla K_{Tf}^k(t,x,y),
	\end{flalign*}
	then for $T\in(0,t^{-\half}]$, any $a\in(0,1),$
	\begin{align*}|\Rt(x,y)|&\leq C_{a,k}\chi_{B_x}(y)t^{(1-\k')k-\k'-\n}T^{\frac{-2k+4}{3}}\exp(-a tT^2 h(x,y))\exp(-\frac{ad^2(x,y)}{4t}).
	\end{align*}
	Here $C_{a,k}$ is a constant depends on $a,k,$ $\k'=\frac{\k+2}{3}$ (from Proposition \ref{asyexp}), $B_x=\{y\in M: d(x,y)<{\tau}\},$ and $\chi_{B_x}(y)$ denotes the characteristic function of $B_x$.
\end{lem}
\begin{proof}
	Since the support of $\Delta\phi(x,y)$ and $\nabla \phi(x,y)$ is a subset of $\{(x,y)\in M\times M:\frac{d^2(x,y)}{\tau^2}\in (\half,1)\},$ by Proposition \ref{asyexp}, Lemma \ref{vexpv} and the fact that $0<a<1$,
\begin{flalign*}&|\Delta\phi(x,y)K^k_{Tf}(t,x,y)+(\nabla\phi(x,y),\nabla K_{Tf}^k(t,x,y))|\\
&\leq C_{k,a}\chi_{B_x}t^{-\n}\exp(-\frac{(1-a)d^2(x,y)}{4t})\exp(-atT^2h(x,y))\exp(-\frac{ad^2(x,y)}{4t})\\
&\leq C_{k,a,k} \chi_{B_x} t^{(1-\k')k-\k'-\n}\exp(-atT^2h(x,y))\exp(-\frac{ad^2(x,y)}{4t}).
\end{flalign*}
The last inequality follows form the fact that the function $t^l\exp(-t)\leq C_l$ for $t\in(0,\infty),l>0.$	

Similarly, by Proposition \ref{asyexp}, Lemma \ref{vexpv} and the fact that  $tT^2\leq 1$, we have
	\begin{flalign*}
	|t^k \phi(x,y)R_{k,T}|&\leq C_{k}\chi_{B_x}\sum_{j=k}^{k+2}t^{j-\n}T^{4(j+1)/3}\bar{V}_\gamma^{\k' (j+1)}\exp(-t\,h_T(x,y))\exp(-\frac{d^2(x,y)}{4t}) \\
	&\leq C_{k}'\chi_{B_x}t^{k-\n}T^{\frac{4(k+1)}{3}}\bar{V}_\gamma^{\k' (k+1)}\exp(-{tT^2}h(x,y)))
	\exp(-\frac{d^2(x,y)}{4t})\\
	&\leq C_{a,k}\chi_{B_x}t^{(1-\k')k-\k'-\n}T^{\frac{-2k+4}{3}}\exp(-{atT^2}h(x,y)))
	\exp(-\frac{ad^2(x,y)}{4t}).
	\end{flalign*}
This finishes the proof.	
\end{proof}

\section{Parabolic Distance and Heat Kernel Estimate}\label{heatker}
With the construction of the parametrix and the error estimate in the last section, we are now faced with the task of proving that it gives the desired asymptotic expansion of heat kernel. To this end, we need to estimate the convolutions of these terms, which seem quite daunting. Remarkably we found that a parabolic distance that appeared previously in Li-Yau's famous work \cite{liyau1986} on the Harnack estimate of the heat kernel of Schr\"odinger operators greatly simplifies the task, both computationally and conceptually. Our inspiration actually comes from the path integral formalism of quantum mechanics.

Another remarkable feature of the parabolic distance is its connection with the Agmon distance \cite{agmon2014lectures}, \cite{bismutzhang1992cm}, \cite{DY2020cohomology}, which we will use to establish the needed lower bound for the parabolic distance.
The resulting pointwise asymptotic expansion of the heat kernel will then be strong enough to pass to the trace of the heat kernel in the noncompact setting.

Let $K_{Tf}^k$ be the parametrix of $\partial_t + \Box_{Tf}$ constructed in Section \ref{expansion}, i.e.
\[K_{Tf}^k(t,x,y)=\phi(x,y)\e_0(t,x,y)\e_{1,T}(t,x,y)\sum_{j=0}^kt^j\t_{T,j}(x,y),\]
where $\phi$ is the cut-off function defined in (\ref{cutoff}).

We define convolution of $f(t,x,y),g(t,x,y)\in\Gamma (E)$ as
\[(f*g)(t,x,y)=\int_0^t\int_M(f(t-s,x,z), g(s,z,y))_zdvol(z)ds.\]

Let $K_{Tf}$ denote the heat kernel of $\Box_{Tf}$. By the Duhamel Principle, we  have
\begin{lem} \label{duhamel} The heat kernel $K_{Tf}$ is given by
	\[K_{Tf}(t,x,y)=K^k_{Tf}(t,x,y)+(K^k_{Tf}*\sum_{l=1}^\infty(-1)^l(\Rt)^{*l})(t,x,y).\]
	Here \[\Rt^{*l}=\underbrace{\Rt*...*\Rt}_{\text{$l$ times}}.\]
\end{lem}

Motivated by the path integral formalism of quantum mechanics, for any piecewise smooth curve $c:[0,t]\mapsto M$, s.t. $c(0)=x, c(t)=y$, we define
\[S_{t,x,y}(c)=\int_0^t\left( \frac{|c'(s)|^2}{4}+T^2V(c(s)) \right) ds.\] 

Let $C_{t,x,y}:=\{ \, c:[0,t]\mapsto M \ | \ c \ \mbox{is piecewise smooth},\ c(0)=x,c(t)=y \, \}$.
Define the following parabolic (meta-)distance  \begin{equation}\label{d0}\d(t,x,y):=\inf_{c\in C_{t,x,y}}S_{t,x,y}(c).\end{equation} 

The following lemma summarizing its fundamental properties follows mostly from the definition.

\begin{lem}\label{paradist}
$\d(t,x,y)$ is a parabolic (meta-)distance; that is
\begin{itemize}
	\item $\d(t,x,y)\geq 0$;
	
	\item $\d(t,x,y)=\d(t,y,x)$;
	
	\item for $0\leq s\leq t$, we have
	\begin{equation}\label{d2}\d(t-s,x,y)+\d(s,y,z)\geq \d(t,x,z).\end{equation}
\end{itemize}
Moreover,  
\begin{equation}\label{d1}\d(t,x,y)\leq \frac{d^2(x,y)}{4t}+t\,h_T(x,y).\end{equation}

\end{lem}

The last inequality follows from taking a minimal geodesic $\tilde{c}$ connecting $x$ and $y$ and noting that $S_{t,x,y}(\tilde{c})=\frac{d^2(x,y)}{4t}+t\,h_T(x,y)$.
 The inequality (\ref{d1}) connects the parabolic distance to our parametrix.

Conceptually the most crucial property of the parabolic distance for the estimation of the convolutions of the error terms is the triangle inequality (\ref{d2}). We illustrate this by an example. If $V_T=0$, then $\d(t,x,y)= \frac{d^2(x,y)}{4t}$. In this case, the triangle inequality (\ref{d2}) reduces to the well known
	\[ \frac{d^2(x,y)}{t-s} +\frac{d^2(y,z)}{s} \geq \frac{d^2(x,z)}{t},\]
which plays a crucial role in the classical asymptotic expansion for the heat kernel.

The following lemma will be also needed in the heat kernel estimate involving the convolutions, and whose proof follows from a standard argument of volume comparison.

\begin{lem}\label{metriceq00}
	For $x\in M$ and $\delta<\tau$, denote $B_x=\{y\in M:  d(x,y)<{\delta}\}.$
	Then there exists $A=A(F_0,\tau,\delta,n)>0$, s.t.
	\begin{align*}
	\int_{B_x}\exp(-\frac{d^2(x,z)}{t})dz\leq A t^\n.
	\end{align*}
	Recall that $F_0$ is the curvature bound, $\tau$ is the injectivity radius bound.
\end{lem}

With these preparations we now turn to the estimation of the convolution terms in the Duhamel Principle, Lemma \ref{duhamel}. From now on, we fix an integer $k$ sufficiently large so that 
$$\alpha(k, \k, n)=\frac{1}{3}(1-\k)k -\frac{\k +2}{3} -\n +1 >0. $$ 

\begin{lem} \label{convolestimate1}
Assume that $t\in(0,1]$ and $T\in(0,t^{-\half}]$. Then for any $a\in (0, 1)$, 
there exist $C=C(k,a, \k, \tau,F_0) >0$, such that, for all $l\in \mathbb N$, 
\[|K_{Tf}^k*\Rt^{*l}|(t,x,y) \leq\frac{C^lt^{\alpha l}T^{\beta l}}{l!} \exp(-a\d(t,x,y)),
\]
where $\alpha=\alpha(k, \k, n)$ as above, and $\beta=\beta(k) = \frac{-2k+4}{3}$.
\end{lem}
\begin{proof}
Let $B_x=\{y\in M:d(x,y)<\tau\}$, then by the volume comparison, we have 
\begin{equation}\label{volume}vol(B_x)\leq C_\tau.\end{equation}

From Lemma \ref{estRk} and (\ref{d1}), we have for $T\in(0,t^{-\half}]$, any $a\in(0,1),$
\begin{align*}|\Rt(x,y)|&\leq C_{a,k}\chi_{B_x}(y)t^{\alpha-1 }T^{\beta}\exp(-a \d(t,x,y)).
\end{align*}
	
Therefore by (\ref{d2}), 
\begin{flalign*}
	|\Rt^{*l}|&=\left|\int_0^t\int_0^{t_1}...\int_0^{t_{l-2}}\int_M...\int_M\Rt(t-t_1,x,z_1)\Rt(t_1-t_2,z_1,z_2) \right. \\
	& \left. \times \Rt(t_2-t_3,z_2,z_3) \cdots \Rt(t_{l-1},z_{l-1},y)dvol(z_{l-1}) \cdots dvol(z_1)dt_{l-1}dt_{l-2} \cdots dt_1 \right|\\
	&=\left|\int_0^t\int_0^{t_1}...\int_0^{t_{l-2}}\int_{B_x}...\int_{B_{z_{l-2}}}\Rt(t-t_1,x,z_1)\Rt(t_1-t_2,z_1,z_2) \right. \\
	&\times \left. \Rt(t_2-t_3,z_2,z_3) \cdots \Rt(t_{l-1},z_{l-1},y)dvol(z_{l-1}) \cdots dvol(z_1)dt_{l-1}dt_{l-2}...dt_1 \right|\\
	&\leq (C_\tau C_{a,k})^l T^{\beta l} \exp(-a\d(t,x,y)) \int_0^t\int_0^{t_1}...\int_0^{t_{l-2}}(t-t_1)^{\alpha -1} \cdots t_{l-1}^{\alpha-1} dt_{l-1} \cdots dt_1\\
	&\leq \frac{C^lt^{\alpha l-1}T^{\beta l}}{(l-1)!} \exp(-a\d(t,x,y)).
	\end{flalign*}

On the other hand, by Proposition \ref{asyexp},  $K_k(t,x,y)\leq C t^{-\n}\exp(-a'\d(t,x,y))$, where, for our purpose, $a'\in (0, 1)$ is chosen to be $a'=\frac{1+a}{2}=a+b$, with $b=\frac{1-a}{2}>0$. Hence

\begin{align*}|K_{Tf}^k* \Rt^{*l}|(t,x,y )&\leq \frac{C^{l+1}T^{\beta l}}{(l-1)!}\int_0^t\int_{B_x}(t-s)^{-\n}s^{\alpha l-1}\\
& \times\exp(-a'\d(t-s,x,z))\exp(-a\d(s,z,y))dvol(z)ds\\
&\leq \frac{C^{l+1}T^{\beta l}}{(l-1)!} \exp(-a\d(t,x,y)) \\
& \times
\int_0^t s^{\alpha l-1} \int_{B_x}(t-s)^{-\n} \exp(-\frac{b d^2(x,y)}{4(t-s)}) dvol(z)ds\\
&\leq \frac{A C^{l+1} t^{\alpha l} T^{\beta l}}{(\alpha l) (l-1)!} \exp(-a\d(t,x,y)). 
\end{align*}


Here in the last inequality, we have made use of Lemma \ref{metriceq00}.
\end{proof}

We summarize our discussion so far.

\begin{thm} \label{heatasym} The heat kernel $K_{Tf}$ has the following complete pointwise asymptotic expansion. For any $x, y\in M$ such that $d(x, y) \leq 1/2\tau$, 
	\[K_{Tf}(t,x,y) \sim \frac{1}{(4\pi t)^{\n}}\exp(-d^2(x,y)/4t)\exp(-t\, h_T(x,y)) 
	\sum_{j=0}^{\infty}t^j\t_{T, j}(x,y),  \]
as $t\rightarrow 0$. 
Each $\t_{T,j}$ is a polynomial of $T$:
\[\t_{T,j}(x,y)=\sum_{l=0}^{[\frac{j}{3}]+j}T^l\t_{l,j}(x,y),\]
and, when restricted to the diagonal of $M\times M$, $\t_{l,j}(y,y)$ can be written as an algebraic combination of the curvature of the metric $g$, the function $f$, as well as their derivatives, at $y$; in addition, $\t_{T,0}(y,y)=\operatorname{Id}.$
Moreover, we have the following remainder estimate. For any $k$ sufficiently large and any
$a\in (0, 1)$,  
\begin{align*}
\left| K_{Tf}(t,x,y) - \frac{1}{(4\pi t)^{\n}}\exp(-d^2(x,y)/4t)\exp(-t\, h_T(x,y)) 
\sum_{j=0}^{k}t^j\t_{T, j}(x,y) \right| \\
\leq Ct^{\frac{1}{3}(1-\k)k -\frac{\k +2}{3} -\n +1}T^{\frac{-2k+4}{3}}\exp(-a \d(t,x,y)) , \hspace{1.5in}
\end{align*} 
for $t\in(0,1]$ and $T\in(0,t^{-\half}]$. 
\end{thm}

\begin{rem}
Here the choice for $t\in(0,1]$ and $T\in(0,t^{-\half}]$ is for simplicity and convenience. Our discussion works for $t\in(0,t_0]$ and $T\in(0, T_0t^{-\half}]$ but the estimates will depend on those choices as well. 
\end{rem}

Without an effective lower bound on the parabolic distance $\d(t,x,y)$ in our noncompact setting, the pointwise asymptotic expansion for the heat kernel of the Witten Laplacian will not be very useful beyond recovering the classical expansion. In particular, in passing from the pointwise asymptotic expansion to the asymptotic expansion of the (global) heat trace, we need remainder estimates which can compensate for the divergent volume integral. 
Here we explore the interesting connection of the parabolic distance to the Agmon distance and establish such an effective lower bound.

Recall that, in our setting, the Agmon metric (Cf. \cite{agmon2014lectures}, \cite{bismutzhang1992cm}, \cite{DY2020cohomology}) is $T^2|\nabla f|^2g$. For any piecewise smooth curve $c$ in $M$, denote $L_{Tf}(c)$ the Agmon length of $c$, i.e.,  the length of $c$ with respect to Agmon metric $T^2|\nabla f|^2g.$

First of all, we note

\begin{lem}\label{paradagmon}
Let $c\in C_{t,x,y}$ be a piecewise smooth curve.Then, 
\begin{equation}\label{agmon1} S_{t,x,y}(c)\geq L_{Tf}(c).\end{equation}	
\end{lem}
\begin{proof}
This follows from an elementary inequality as
\[ S_{t,x,y}(c)=\int_0^t\frac{|c'(s)|^2}{4}+T^2V(c(s))ds\geq \int_0^t T|\nabla f|(c(s))|c'(s)|ds=L_{Tf}(c).\]
Thus the parabolic distance is bounded from below by the Agmon distance (but we actually will be using the Agmon length later).
\end{proof}

The following lemma says that the Agmon length can be bounded from below effectively if the potential function varies considerably along a curve.

\begin{lem}\label{agmon}
Let $c\in C_{t,x,y}$ be a piecewise smooth curve. If $$\inf_{s\in[0,t]}V(c(s))\leq \frac{1}{2}\sup_{s\in[0,t]}V(c(s)), $$
then there exists constant $\bar{\beta}>0$ depending only on the bounds in the tameness condition, such that 
$$ L_{Tf}(c) \geq \bar{\beta} T\sup_{s\in[0,t]}|V|^{1-\k}(\gamma(s)).$$
\end{lem}
\begin{proof} Set $\bar{V}_c:=\sup_{s\in[0,t]}V(c(s))$.
Then we can find an interval $[a,b]\subset [0,t]$, s.t. $V(c(a))=\frac{\bar{V}_c}{2},$ $V(c(b))=\bar{V}_c$ (or vice versa, $V(c(b))=\frac{\bar{V}_c}{2},$ $V(c(a))=\bar{V}_c$). Moreover, for all $s\in[a,b]$, $V(c(s))\geq \frac{\bar{V}_c}{2}.$ 

Now by the $\kappa$-regular tame condition, 
\begin{flalign*}
\frac{\bar{V}_c}{2}&=|V(c(a))-V(c(b))|\leq\int_a^b |\nabla V(c(s))||c'(s)|ds\\
&\leq C \int_a^b |V(c(s))|^{\frac{\k+2}{2}}|c'(s)|ds\\
&\leq C \bar{V}_c^{\frac{\k+1}{2}} \int_a^b|\nabla f|(c(s))|c'(s)|ds\\
&\leq CT^{-1} \bar{V}_c^{\frac{\k+1}{2}} L_{Tf}(c|_[a,b])
\end{flalign*}
Thus, for $\bar{\beta}=\frac{1}{2C}>0$, 
\begin{equation}\label{agmon2}
L_{Tf}(c) \geq L_{Tf}(c|_{[a,b]})\geq \bar{\beta} T{\bar{V}_c}^{\frac{1-\k}{2}}=\bar{\beta} T\sup_{s\in[0,t]}|V|^{\frac{1-\k}{2}}(\gamma(s)).
\end{equation}
\end{proof}


Finally we arrive at the following effective lower bound for the parabolic distance.
\begin{lem} \label{effbound} One has
\begin{equation} \label{effbdparadist}
\d(t,x,y)\geq \min\{\bar{\beta} TV^{\frac{1-\k}{2}}(x),\frac{tT^2V(x)}{2}\}.
\end{equation}
In particular, for $t\in (0, 1], T=t^{-\half}$,
\begin{equation}\label{dagmon}
\d(t,x,y)\geq \bar{\beta} V^{\frac{1-\k}{2}}(x) \min \{1,\frac{V(x)^{\frac{\k+1}{2}}}{2 \bar{\beta} }\}.
\end{equation}
\end{lem}
\begin{proof}
%
Let $\gamma:[0,t]\mapsto M$ be a curve minimizing $S_{t,x,y}.$ As before, set $\bar{V}_\gamma:=\sup_{s\in[0,t]}V(\gamma(s)).$

If $V(\gamma(s))\geq \frac{\bar{V}_\gamma}{2}$ for all $s\in[0,t]$, then we have \begin{equation}\label{d3}\d(t,x,y)\geq \frac{tT^2\bar{V}_\gamma}{2}\geq \frac{tT^2V(x)}{2}.\end{equation}

If not, by Lemma \ref{agmon}, \begin{equation}\label{d4} L_{Tf}(x,y)\geq \bar{\beta} T{\bar{V}_\gamma}^{\frac{1-\k}{2}}\geq \bar{\beta} TV^{\frac{1-\k}{2}}(x).\end{equation}

Therefore, by Lemma \ref{paradagmon},  
$$\d(t,x,y)\geq \min\{\bar{\beta}TV^{\frac{1-\k}{2}}(x),\frac{tT^2V(x)}{2}\}. $$ 
Our results follow. 
\end{proof}

We also note the following lemma which was used in the previous section.
\begin{lem}\label{vexpv}
For $a\in(0,1),t\in(0,1)$, $l>0$, there exists $C_{a,\k,l}>0$, s.t.
\[\bar{V}_\gamma^l\exp(-\frac{d^2(x,y)}{4t})\exp(-tT^2h(x,y))\leq C_{a,\k,l}t^{-l}T^{-2l}\exp(-\frac{ad^2(x,y)}{4t})\exp(-atT^2h(x,y)),\]
where $\gamma$ is the minimal geodesic connecting $x$ and $y$, $\bar{V}_\gamma=\sup_{p\in\gamma}|V(p)|.$
\end{lem}
\begin{proof}
When $\inf_{p\in\gamma}|V(p)|\geq \frac{\bar{V}_\gamma}{2}$, $h(x,y)\geq \frac{\bar{V}_\gamma}{2}$, hence $\bar{V}_\gamma^l\exp(-(1-a)tT^2h(x,y))\leq C_{a,l}t^{-l}T^{-2l}$ for some $C_{a,l}>0.$

Otherwise, by Lemmas \ref{paradagmon} and \ref{agmon}, $\frac{d^2(x,y)}{4t}+tT^2h(x,y)\geq \bar\beta T\bar{V}_\gamma^{\frac{1-\k}{2}}\geq \bar\beta(T^2 \bar{V}_\gamma)^{\frac{1-\k}{2}}$. Therefore, there exist $C_{a,\k,l}$ such that
\begin{align*}
\bar{V}_\gamma^l\exp(-(1-a)\frac{d^2(x,y)}{4t})\exp(-(1-a)tT^2h(x,y)) & \leq\bar{V}_\gamma^l\exp(-(1-a)\bar\beta (T^2V)^{1-\k}) \\
&\leq C_{a,\k,l}T^{-2l}\leq C_{a,\k,l}t^{-l}T^{-2l}\end{align*}
which yields the result.
\end{proof}
%
%

Combining the above discussion with Theorem \ref{heatasym} we have 

\begin{thm} \label{dheatasym} For $T=t^{-\half}$, the heat kernel $K_{t^{-\half}f}$ of the Witten Laplacian has the following complete pointwise (diagonal) asymptotic expansion. For any $x\in M$, 
	\[K_{t^{-\half}f}(t,x,x) \sim \frac{1}{(4\pi t)^{\n}}\exp(- |\nabla f|^2(x)) 
	\sum_{j=0}^{\infty}\sum_{l=0}^{[\frac{j}{3}]+j}t^{j-\frac{l}{2}}\t_{l,j}(x,x),  \]
	as $t\rightarrow 0$. 
	Moreover, for any $k$ sufficiently large and any
	$a\in (0, 1)$,  
	\begin{align*}
	\left| K_{t^{-\half}f}(t,x,x)- \frac{1}{(4\pi t)^{\n}}\exp(- |\nabla f|^2(x)) 
	\sum_{j=0}^{\infty}\sum_{l=0}^{[\frac{j}{3}]+j}t^{j-\frac{l}{2}}\t_{l,j}(x,x) \right|  \\
	\leq Ct^{\frac{1}{3}(2-\k)k -\frac{\k +1}{3} -\n}\exp(-a\bar{\beta} |\nabla f|^{1-\k}(x) ) , \hspace{.8in}
	\end{align*} 
	for $t\in(0,1]$ and $x\in M$. In particular, we have the following small time asymptotic expansion of the heat trace:
		\[ \operatorname{Tr}\left(\exp (-t\Box_{t^{-\frac{1}{2}}f})\right) \sim \frac{1}{(4\pi t)^{\n}} 
	\sum_{j=0}^{\infty} \sum_{l=0}^{[\frac{j}{3}]+j}t^{j-\frac{l}{2}} \int_M \exp(-|\nabla f|^2(x)) \operatorname{tr}(\t_{l, j}(x,x)) dx,  \]
	as $t\rightarrow 0$, with the remainder estimate
	\begin{align*}
	\left| \operatorname{Tr}\left(\exp (-t\Box_{t^{-\frac{1}{2}}f})\right)- \frac{1}{(4\pi t)^{\n}} 
	\sum_{j=0}^{k} \sum_{l=0}^{[\frac{j}{3}]+j}t^{j-\frac{l}{2}} \int_M \exp(-|\nabla f|^2(x)) \operatorname{tr}(\t_{l, j}(x,x)) dx \right|  \\
	\leq Ct^{\frac{1}{3}(2-\k)k -\frac{\k +1}{3} -\n }. \hspace{1.5in}
	\end{align*} 
\end{thm}

\begin{proof}
This follows from Theorem \ref{heatasym} and Lemma \ref{effbound} by noting that $V\geq (2 \bar{\beta})^\frac{2}{\k+1} $ outside a compact set.
\end{proof}

\section{Local Index Theorem for Witten Laplacian} 

We now turn to the local index theorem for the Witten Laplacian. From the discussion at the end of Section 2 (see \eqref{stohk} and after) we have 
	 \begin{equation} \label{stohkc}
\chi(M,d_{Tf})=\sum_{i=0}^n (-1)^i dim(H^i_{(2)}(M,d_{Tf}))=\operatorname{Tr}_s(\exp(-t\Box_{Tf}))
\end{equation}
is independent of $t.$  Moreover,  by Theorem 1.3 in \cite{DY2020cohomology}, $\chi(M,d_{Tf})$ is independent of $T>0$. As a consequence, Theorem \ref{dheatasym} reduces the index formula for Witten Laplacian to a local index theorem, which we will develop in this section. 

First we summarize what we know about the index of the Witten Laplacian as the following McKean-Singer type formula.

\begin{prop} \label{locind1} Assume that $(M, g, f)$ is polynomial tame. Then for $T>0$,
$\chi(M,d_{Tf})$ is independent of $T$ and
\[\chi(M,d_{Tf})= \int_M \operatorname{tr}_s(K_{Tf}(t,x,x)) dx\]
for any $t>0$. Here $dx$ denotes the volume form induced by $g$.
\end{prop}

In the usual approach to the local index theorem, one studies the integrand, the pointwise supertrace $\operatorname{tr}_s(K_{Tf}(t,x,x))$,  in the limit $t\rightarrow 0$ via the Getzler's rescaling. 
To proceed with Getzler's rescaling technique, we now fix $x_0\in M$ and let $x$ be the normal coordinates near $x_0.$ Thus $x=0$ at $x_0$, and we will use $0$ and $x_0$ interchangeably in this section. We trivialize the bundle $\Lambda^*(M)$ in the normal neighborhood $U$ by parallel transport along radical geodesic from $x_0.$ In fact, we can assume $M=T_{x_0}M$ for now by extending everything trivially outside the normal neighborhood (we will see that we can localize the problem because of Theorem \ref{dheatasym}).

For usual Getzler's rescaling techniques (a la Bismut-Zhang \cite{bismutzhang1992cm} for the de Rham complex), one defines $\delta_\epsilon$ as follows:
\begin{enumerate}
	\item For function $f\in C^\infty([0,\infty)\times U)$, $(\de f)(t,x)=f(\ep t,\epsq x)$. As a consequence, we have
	\[\lim_{t\to0}f(t,0)=\lim_{\ep\to0}(\de f)(t,0).\] Moreover, $\de f(t,x)\den=f(\ep t,\epsq x),$ $\de \partial_{x_i}\den=\epnsq \partial_{x_i},$ $\de \partial_t\den=\ep \partial_t.$
	\item Let $\{e_i\}_{i=1}^n$ be a local frame near $x_0,$ $\{e^i\}_{i=1}^n$ its dual frame. Then for $c(e_i)=e^i\wedge-\iota_{e_i},\hc(e_i)=e^i\wedge+\iota_{e_i}$, we define $\de c(e_i)=\epnq e^i\wedge-\epq \iota_{e_i}$,$\de \hc(e_i)=\epnq e^i\wedge+\epq \iota_{e_i}.$ Now let $\ce(e_i)=\epnq e^i\wedge-\epq \iota_{e_i},\hce(e_i)=\epnq e^i\wedge+\epq \iota_{e_i}$, then $\de c(e_i)\den=\ce(e_i), \de\hc(e_i)\den=\hce(e_i).$
\end{enumerate}
Recall that $K_{Tf}$ is the heat kernel of $\Box_{Tf}$. Then $K_{Tf,\ep}'=\ep^{\n}\de K_{Tf}$ is the heat kernel for $\Box'_{Tf,\ep}:=\ep\de \Box_{Tf}\den.$ Moreover, for small $\epsilon$ \cite[(4.60)]{bismutzhang1992cm}
\[\Box'_{Tf,\ep}=-\Delta_{T_{x_0}M}\operatorname{Id}_{\Lambda^*T^*_{x_0}M}+\frac{1}{2}\sum_{i<j<k<l}R_{ijkl}(0)e^i\wedge e^j\otimes\hat{e}^k\wedge\hat{e}^l +O(\epsq),\]
where $\Delta_{T_{x_0}M}$ is the Euclidean Laplacian on $T_{x_0}M$, 
  and $R_{ijkl}(x)$ is the Riemannian curvature tensor at $x$. 

This is the usual Getzler's rescaling. As $\ep\to 0$, the information of $f$ disappears. But for the noncompact case, unlike the compact case, the index should depend on $f$. To deal with this issue, we introduce the following rescaling technique: we let $T$ join the game. 

As mentioned before, the index $\chi(M,d_{Tf})=\operatorname{Tr}_s(\exp(-t\Box_{Tf}))$ is independent of $T>0$. Hence, in our rescaling, we define, in addition,  $\de (T)=\epnsq T.$

Now under new rescaling, then we have
\begin{lem}
	Let $\Box_{Tf,\ep}:=\ep\de\Box_{Tf}\den$. Then
	\[\Box_{Tf,0}:=\lim_{\ep\to0} \Box_{Tf,\ep}=-\Delta_{T_{x_0}M}\operatorname{Id}_{\Lambda^*T^*_{x_0}M}-\frac{1}{2}\sum_{i<j<k<l}R_{ijkl}(0)e^i\wedge e^j\otimes\hat{e}^k\wedge\hat{e}^l+V_T(x_0)+TL_{f,0}.\]
	Here $L_{f,0}=\nabla^2_{e_i,e_j}f(x_0)e_i\otimes\hat{e}_j.$
\end{lem}

\begin{proof} By Proposition \ref{estihdg}, $\Box_{Tf}=\Delta-TL_f+T^2|\nabla f|^2.$
	By \cite[(4.60)]{bismutzhang1992cm}, 
	\[\ep\de\Delta\den=-\Delta_{T_{x_0}M}\operatorname{Id}_{\Lambda^*T^*_{x_0}M}+\frac{1}{2}\sum_{i<j<k<l}R_{ijkl}(0)e^i\wedge e^j\otimes\hat{e}^k\wedge\hat{e}^l +O(\epsq).\]
	On the other hand, by the new rescaling in $T$, $\ep\de(T^2|\nabla f|^2)\den= T^2|\nabla f|^2(x_0)+O(\epsq)$. Now
	$L_f=\nabla^2_{e_i,e_j}f[e^i\wedge,\iota_{e_j}]=-\nabla^2_{e_i,e_j}f c(e^i) \hat{c}(e^j) .$ Hence $$\ep \de (TL_f) \den = -T\nabla^2_{e_i,e_j}f(x_0)e_i\otimes\hat{e}_j + O(\epsq).$$
	Our result follows. 
\end{proof}

Denote $\tilde{R}(x_0)=-R_{ijkl}(x_0)e^i\wedge e^j\otimes\hat{e}^k\wedge\hat{e}^l$.
Let $K_{Tf,0}$ be the heat kernel of $\Box_{Tf,0}.$ Clearly $-\Delta_{T_{x_0}M}\operatorname{Id}_{\Lambda^*T^*_{x_0}M}$ commutes with $\frac{\tilde{R}(x_0)}{2}+TL_f(x_0)+V_T(x_0)$. Therefore we have 
\begin{align} \label{lhk} K_{Tf,0}=\e_0\exp(-t[\frac{\tilde{R}(x_0)}{2}+TL_f(x_0)+V_T(x_0)]).\end{align} 

By Theorem \ref{dheatasym}, $K_{Tf}(t,x,x)$ has the following asymptotic expansion,
\[K_{Tf}(t,x,x)=(4\pi t)^{-\n}\exp(-tV_T)\sum_{j=0}^\infty t^j\t_{T,j}(x,x), \]
with strong remainder estimate when $T=t^{-1/2}$.
In particular, 
\begin{align} \label{dhk}
K_{t^{-\half}f}(t,x,x)=(4\pi t)^{-\n}\exp(-V)\sum_{k\in \frac{1}{2} \mathbb N} t^k\sum_{j-\half l=k,l\leq j +[\frac{j}{3}]}\t_{l,j}(x,x).
\end{align} 
Here $\mathbb N$ denotes the set of natural numbers which by our convention contains $0$. 
Thus we can upgrade Proposition \ref{locind1} to 
\begin{prop} \label{locind2} For $T>0$, 
\begin{align} \chi(M,d_{Tf}) & =\lim_{t\to0}\operatorname{Tr}(\exp(-t\Box_{t^{-\half}f})) = \int_M \lim_{t\to0} \operatorname{Tr}_s^{\Lambda^*(TM)} (K_{t^{-\half}f}(t,x,x)) dx\\
& =\frac{1}{(4\pi)^\n}\int_{M}\exp(-|\nabla f|^2)\sum_{j-\half l=\n}\operatorname{tr}_s^{\Lambda^*(TM)}(\t_{l,j}(x,x)) dx. \nonumber \end{align}
Here (to emphasize) we use $\operatorname{tr}_s^{\Lambda^*(TM)}$ to denote the pointwise supertrace on $\Lambda^*(TM)$ which was previously denoted by $\operatorname{tr}_s$.
\end{prop}

Now for $I=\{i_1,...,i_k\}\subset\{1,2,...,n\},(i_1<...<i_k)$, denote $c(e_I)=c(e_{i_1})...c(e_{i_k}),\hc(e_I)=\hc(e_{i_1})...\hc(e_{i_k}).$ Write $\t_{l,j}=\sum_{I,J\subset\{1,2,...,n\}}\t_{l,j,I,J}c(e_I)\hc(e_J).$
The following Proposition on the key property of the supertrace is well known.

\begin{prop} \label{supertrace}
	For $I,J\subset\{1,2,...,n\}$,
	\begin{equation*}
	\operatorname{tr}_s^{\Lambda^*(TM)}\left(c(e_I)\hc(e_J)\right)=\begin{cases}
	(-1)^{\frac{n(n+1)}{2}}2^n, \mbox{ if $I=J=\{1,2,...,n\}$}\\
	0, \mbox{ otherwise.}
	\end{cases}
	\end{equation*}
	
\end{prop}

Thus $\operatorname{tr}_s^{\Lambda^*(TM)}(\t_{l,j})=(-1)^{\frac{n(n+1)}{2}}2^n\t_{l,j,I_n,I_n}$, where $I_n=\{1,2,...,n\}.$ We now recall the Berezin integral formalism. For any $\o\in \O^*(TM)\hat\otimes \O^*(TM)$, $I\subset {1,2,...,n}$, we can write $\o$ as
\[\o:=\sum_{I}w_I\hat{e}^I.\]
Then the Berezin integral $\int^B:\O^*(TM)\hat\otimes \O^*(TM)\mapsto \O^*(TM)$ is defined as
\begin{equation*}
\int^B \o=\o_{I_n}.
\end{equation*}

The following lemma is also  well known in local index theory and the Getzler rescaling technique.

\begin{lem}
	We have 
	\begin{align} \label{locind0}
	\lim_{t\to 0}\operatorname{tr}_s^{\Lambda^*(TM)}(K_{t^{-\frac{1}{2}}f})(t,x_0,x_0)dx=(-1)^{\frac{n(n+1)}{2}}2^n \int^B \lim_{\ep\to0} \ep^\n(\de K_{t^{-\half}f})(t,x_0,x),\end{align}
	provided that the right hand limit exists. 
\end{lem}

\begin{proof} Write $K_{t^{-\half}f}(t,x_0,x)=\sum_{I,J\subset\{1,2,...,n\}}a_{I,J}(t, x) c(e_I)\hc(e_J)$. By Proposition \ref{supertrace},
	\[ \operatorname{tr}_s^{\Lambda^*(TM)}(K_{t^{-\half}f}(t,x_0,x_0))= (-1)^{\frac{n(n+1)}{2}}2^n  a_{I_n, I_n}(t, x_0).	
	\]
On the other hand, 
\[(\ep^\n\de K_{t^{-\half}f})(t,x_0,x)= \sum_{I,J\subset\{1,2,...,n\}}a_{I,J}(\ep t, \ep^{\half} x) \ep^\n c_\ep(e_I)\hc_\ep(e_J).
\]
Hence,
\[ \int^B \lim_{\ep\to0} \ep^\n(\de K_{t^{-\half}f})(t,x_0,x) = \lim_{\ep\to0}  a_{I_n, I_n}(\ep t, \ep^{\half} x) e^1\wedge \cdots \wedge e^n = \lim_{t\to0}  a_{I_n, I_n}(t, x_0) dx.
\]
Our result follows.
\end{proof}

For the right hand side of the previous lemma, we have the following proposition.
\begin{prop}\label{epsilongoesto0} There exists $a\in(0,1)$ such that
$$|\ep^\n(\de K_{t^{-\half}f})(t,x,x)-K_{t^{-\half}f,0}(t,x,x)|\leq C\ep t^{2-\k-\n}\exp(-a V^{1-\k}).$$
\end{prop}

\begin{proof} 
	
	Let $K_0(t,x,y)=\phi(x,y)K_{Tf,0}(t,x,y).$ Then by the tameness condition, for some $a\in(0,1)$ we have 
	\[|(\Box_{Tf,\ep}-\Box_{Tf,0})K_0(t,x,y)|\leq C\chi_{B_x}(y)\ep t^{-\frac{n+1-\k}{2}}T^{-2}\exp(-\frac{ad(x,y)}{4t})\exp(-a tT^2 V(x)).\]
	By the Duhamel principle, 
	\[\ep^\n K_{Tf,\ep}-K_0=(\ep^\n K_{Tf,\ep})*((\Box_{Tf,\ep}-\Box_{Tf,0})K_0(t,x,y)).\] 
	
On the other hand,  $\ep^\n K_{Tf,\ep}=\ep^\n(\de K_{Tf}^k+\sum_{l=1}^\infty \de(K_{Tf}^k*\Rt^{*l}),$ and it is straightforward to check that
	\[|\ep^\n\de K_{Tf}(t,x,y)|\leq C\chi_{B_x}t^{-\n}\exp(-\frac{atT^2}{2}V(x))\exp(-\frac{ad^2(x,y)}{4t}).\]
	Proceeding as in the previous section we finish the proof of the Proposition. 
\end{proof}

Finally, we arrive at our local index theorem for the Witten Laplacian. Recall that $\widetilde{R},\widetilde{\nabla}^2f\in \O^*(M)\hat\otimes \O^*(M)$ are defined as (we abuse the notatin here by omitting the wedge product signs)
\[\widetilde{R}(x)=R_{ijkl}(x) e^ie^j\hat{e}^k\hat{e}^l, \ \ \ \  \widetilde{\nabla}^2f(x)=\nabla^2_{e_i,e_j} f(x) e^i\hat{e}^j.\]

\begin{thm} \label{locind}
For any $x_0\in M$, we have 
\[ \lim_{t\to 0}\operatorname{Tr}_s^{\Lambda^*(TM)}(K_{t^{-\frac{1}{2}}f})(t,x_0,x_0)= \frac{(-1)^{[\frac{n+1}{2}]}}{\pi^\n}\exp(-|\nabla f(x_0)|^2)\int^B\exp(-\frac{\widetilde{R}(x_0)}{2}-\widetilde{\nabla}^2 f(x_0)).
\]
In particular, for $T>0$,
\[  \chi(M,d_{Tf})=\frac{(-1)^{[\frac{n+1}{2}]}}{\pi^\n} \int_M \exp(-|\nabla f|^2)\int^B\exp(-\frac{\widetilde{R}}{2}-\widetilde{\nabla}^2 f).
\]
\end{thm}

\begin{proof}
	By \eqref{locind0} and Proposition \ref{epsilongoesto0}, 
	\begin{align*}
\lim_{t\to 0}\operatorname{tr}_s^{\Lambda^*(TM)}(K_{t^{-\frac{1}{2}}f})(t,x_0,x_0) dx & =(-1)^{\frac{n(n+1)}{2}}2^n \int^B	\lim_{\ep\to0}  (\ep^\n(\de K_{t^{-\half}f})(t,x_0,x)) \\
&=(-1)^{\frac{n(n+1)}{2}}2^n \int^B K_{t^{-\half}f,0} \\
& = \frac{(-1)^{\frac{n(n+1)}{2}}2^n}{(4\pi t)^\n}\int^B  \exp(-t\frac{\tilde{R}(x_0)}{2}-t^\half L_f(x_0)- |\nabla f(x_0)|^2)\\
	&=\frac{(-1)^{[\frac{n+1}{2}]}}{\pi^\n}\exp(-|\nabla f(x_0)|^2)\int^B\exp(-\frac{\tilde{R}(x_0)}{2}-\tilde{\nabla}^2 f(x_0)).\end{align*}
	The second result then follows from Proposition \ref{locind2}.
\end{proof}

\section{Examples From Landau-Ginzburg Models} \label{eflgm}

In this section we will disucss in somewhat detail how our results apply to some examples coming from Landau-Ginzburg models. Some of our discussions benefited from those of \cite{fanfang2016torsion}.

\def\dbar{\bar{\partial}}
Consider a triple $(M, g, f)$, where $(M, g)$ is a K\"ahler manifold with bounded geometry, and $f:\, M \longrightarrow \mathbb C$ a holomorphic function. In this case, one considers the Witten deformation of the $\dbar$-operator 
\[ \dbar_f = \dbar + \partial f \wedge :\, \Omega^k(M, \mathbb C) \longrightarrow \Omega^{k+1}(M, \mathbb C) .
\]
The corresponding Witten Laplacian is then $\Box_{\dbar, f}=\dbar_f^*\dbar_f + \dbar_f\dbar_f^*$. 

On the other hand, one can also consider the underlying real manifold $M$ with the Riemannian metric given by $g$, together with the potential function given by $2\Re f=f + \bar{f}$. It follows from the K\"ahler identity that 
\[ 2\Box_{\dbar, f} = \Box_{2\Re f}.
\]
As a consequence, $\chi(M, \dbar_f)=\chi(M, d_{\Re f})$.

A large class of Landau-Ginzburg models consists of $(\mathbb C^n, g_0, f)$ where $g_0$ is the Euclidean metric and $f: \mathbb C^n \rightarrow \mathbb C$ a so-called nondegenerate quasi-homogeneous polynomial. Here $f\in \mathbb C[z_1, \cdots, z_n]$ is a quasi-homogeneous (also known as weighted homogeneous) polynomial if there are positive rational numbers $q_1, \cdots, q_n$, called the weights, such that 
\[ f(\l^{q_1} z_1, \cdots, \l^{q_n} z_n) =\l f(z_1, \cdots, z_n), 
\]
for all $\l\in \mathbb C^*$. $f$ is called nondegenerate if $f$ contains no monomials of the form $z_iz_j$ for $i\not= j$ and $0$ is the only critical point of $f$ (equivalently, the hypersurface $f=0$ in the weighted projective space is non-singular). By the classification result of \cite{saito1971singularity} (see also \cite[Theorem 3.7]{hertling2012classification}), if $f$ is nondegenerate, then $q_i \leq \half, \forall i$ (and these weights are unique). 

If $f$ is a nondegenerate quasi-homogeneous polynomial, then $(\mathbb C^n, g_0, f)$ (or equivalently, the corresponding real model) is polynomial tame. To see this, one uses a result from \cite{fan2008spinequations}. Indeed, it is shown in \cite[Theorem 5.8]{fan2008spinequations} that if $f$ is a nondegenerate quasi-homogeneous polynomial, then there exists a constant $C>0$ depending only on $f$ such that for all $(u_1, \cdots, u_n) \in \mathbb C^n$, and each $i=1, \cdots, n,$
\begin{align} \label{fjr} |u_i| \leq C \left( \sum_{j=1}^n |\frac{\partial f}{\partial z_j}(u_1, \cdots, u_n)| +1 \right)^{\gamma_i}, 
\end{align}
where $\gamma_i= \frac{q_i}{\min_j(1-q_j)}$.

As $| \nabla \Re f|^2=\sum_j |\frac{\partial f}{\partial z_j}|^2$,  one obtains using the above estimate and quasi-homogeneity that for $m\geq 1$, 
\[ |\nabla^m \Re f| \leq C (| \nabla \Re f |+1)^{\frac{1-m \min_j q_j}{\min_j(1-q_j)}},
\]
where the constant $C$ now also depends on $m$, $n$. Since $q_j \leq \half$, the exponent here
\[ \frac{1-m \min_j q_j}{\min_j(1-q_j)} \leq 2(1-m \min_j q_j).
\]
Thus, if we let $\kappa=\max\{0, 1-4\min_j q_j\}<1$, then the real model here $(\mathbb R^{2n}, g_0, \Re f)$ is $\kappa$-regular tame.

\begin{rem}
	It is also clear from the above discussion that when $m \min_j q_j \geq 1$, we can choose $\kappa=0$, and therefore the real model $(\mathbb R^{2n}, g_0, \Re f)$ is effectively $0$-regular tame.
\end{rem}

Also from the estimate \eqref{fjr} and $q_j \leq \half$ one deduces that 
\[ |z|^2 \leq C(| \nabla \Re f|^2 + 1).
\]
It follows that 
\[\int_{|\nabla\Re f|^2\leq \lambda\}}(\lambda-|\nabla \Re f|^2)^{2n/2}dvol \leq \lambda^{n} \Vol(B(0, \sqrt{C(\l+1)} )) \leq  C'\lambda^{2n}. \] 
And thus $(\mathbb R^{2n}, g_0, \Re f)$ is polynomial tame. Therefore, Theorem \ref{locind} yields the following formula for the Milnor number of $f$, which is stated in \cite{fanfang2016torsion} under additional restriction on the weights of $f$.

\begin{cor}
	If $f\in \mathbb C[z_1, \cdots, z_n]$ is a nondegenerate quasi-homogeneous polynomial, then\[ \chi(\mathbb C^n, \dbar_f)= \frac{(-1)^n}{\pi^n}\int_{\C^n}\exp(-|\partial f|^2)|det(-\partial^2f)|^2dvol.
	\]
\end{cor}

\begin{proof}
	Theorem \ref{locind} applied to the real model $(\mathbb R^{2n}, g_0, \Re f)$ gives us
	\begin{align*} \chi(\mathbb C^n, \dbar_f) & = \chi(\mathbb R^{2n}, d_{\Re f})= \frac{(-1)^{[\frac{2n+1}{2}]}}{\pi^n} \int_{\mathbb R^{2n}} \exp(-|\nabla \Re f|^2)\int^B\exp(-\widetilde{\nabla}^2 \Re f) \\
	& = \frac{(-1)^{n}}{\pi^n} \int_{\mathbb R^{2n}} \exp(-|\nabla \Re f|^2) (-1)^{n}\det(-\nabla^2 \Re f) dvol \\
& = \frac{(-1)^n}{\pi^n}\int_{\C^n}\exp(-|\partial f|^2)|det(-\partial^2f)|^2dvol.	
	\end{align*}
\end{proof}

In the remaining part of the section we discuss the asymptotic expansion of the heat trace for the Witten Laplacian of the Landau-Ginzburg model $(\mathbb C^n, g_0, f)$, or equivalently, its  real model $(\mathbb R^{2n}, g_0, \Re f)$,  for $f$ a nondegenerate quasi-homogeneous polynomial, but without setting $T=t^{-\half}$ as before.

By Theorem \ref{heatasym}, we have a pointwise asymptotic expansion for the heat kernel
with remainder estimate, which we will specialize here on the diagonal. For any $k$ sufficiently large and any
$a\in (0, 1)$,  there exists $C>0$ such that for $t\in(0,1]$ and $T\in(0,t^{-\half}]$,
\begin{align*}
\left| K_{Tf}(t,x,x) - \frac{1}{(4\pi t)^{n}}\exp(-tT^2 V(x)) 
\sum_{j=0}^{k}t^j\t_{T, j}(x,x) \right| \ \ \  \\
\leq Ct^{\frac{1}{3}(1-\k)k -\frac{\k +2}{3} -n +1}T^{\frac{-2k+4}{3}}\exp(-a \d(t,x,x)) . \hspace{.5in}
\end{align*} 
Here 
\[V= |\nabla \Re f|^2 =\sum_j |\frac{\partial f}{\partial z_j}|^2.
\]

We will first see that the remainder estimate is strong enough for the global heat trace, namely it is convergent when integrated on $\mathbb C^n$. By Lemma \ref{effbound}
\[ \d(t,x,x)\geq \min\{\bar{\beta} TV^{\frac{1-\k}{2}}(x),\frac{tT^2V(x)}{2}\}.
\]
On the other hand, by \cite[Lemma 3.11(i)]{fanfang2016torsion}, which follows from the fact that $f$ is a nondegenerate quasi-homogeneous polynomial, 
\[ tV(z_1, \cdots, z_n) \geq V(t^{\delta q_1}z_1, \cdots, t^{\delta q_n}z_n), \ \ \  \ \ \  \delta =\frac{1}{2\min_j(1-q_j)}\leq 1.
\]

Now set
\[ \Omega_t =\left\{ V \leq (\frac{2\bar{\beta}}{tT})^{\frac{2}{1+\k}} \right\}, \ \ \ \ \ \  \Omega_t^c =\mathbb C^n - \Omega_t.
\]
Then on $\Omega_t$,
\[ \d(t,z,z)\geq \frac{tT^2V(z)}{2} \geq \half T^2 V(t^{\delta q_1}z_1, \cdots, t^{\delta q_n}z_n). 
\]
Hence,
\begin{align*}
\int_{\Omega_t} e^{-a \d(t, z, z)} dvol  \leq  \int_{\mathbb C^n} e^{-\half aT^2 V(t^{\delta q_1}z_1, \cdots, t^{\delta q_n}z_n)} dvol = t^{-2\delta |q}| C(a, T), \ \  \ |q| =\sum_j q_j.
\end{align*}

On $\Omega_t^c$, $\d \geq \bar{\beta} TV^{\frac{1-\k}{2}}$. Thus,
\begin{align*}
\int_{\Omega_t^c} e^{-a \d(t, z, z)} dvol  \leq  \int_{\mathbb C^n} e^{-\bar{\beta} TV^{\frac{1-\k}{2}}} dvol =C_1(\bar{\beta}, T).
\end{align*}
And we arrive at
\begin{align*}
\int_{\mathbb C^n} e^{-a \d(t, z, z)} dvol  \leq  t^{-2\delta |q|}| C(a, T) +C_1(\bar{\beta}, T).
\end{align*}

We now look at the terms in the asymptotic expansion given by Theorem \ref{heatasym}. For a multi-index $\alpha=(\alpha_1, \cdots, \alpha_n)$ with $\alpha_i$ nonnegative integer, we denote   
$\partial^\alpha f=\frac{\partial^{|\alpha|}f}{\partial^{\alpha_1}z_1 \cdots \partial^{\alpha_n}z_n}$, $|\alpha|=\alpha_1 + \cdots \alpha_n$. 
From the construction in Section \ref{expansion}, $\t_{T, j}(z, z)$ is a linear combination of  $\partial^{\alpha^1}f \cdots \overline{\partial^{\alpha^l}f}$, with $l\leq j$ and (non-trivial) multi-indeces $\alpha^1, \cdots, \alpha^l$ satisfying $|\alpha^1| + \cdots |\alpha^l| \leq 2j$. 

At this point we make the further assumption that $f$ is homogeneous; namely 
\[ q_1=\cdots = q_n, 
\] and we denote the common value by $q$. Differentiating the equation for quasi-homogeneity gives, 
\[ \l^{q |\alpha |}(\partial^\alpha f)(\l^{q_1} z_1, \cdots, \l^{q_n} z_n) =\l \, \partial^\alpha f(z_1, \cdots, z_n),
\]
from which one deduces that
\[ tV(z_1, \cdots, z_n) = V(t^{\delta q}z_1, \cdots, t^{\delta q}z_n).
\]
Hence,
\begin{align*}
\int_{\mathbb C^n} e^{-tT^2V} \partial^{\alpha^1}f \cdots \partial^{\alpha^l}fdvol  = t^{\delta q \sum_{i=1}^l |\alpha^i| -\delta l - 2n\delta q}\, C_{\alpha^1, \cdots, \alpha^l}(f),
\end{align*}
where $C_{\alpha^1, \cdots, \alpha^l}(f)$ is a constant depending on $f$ and $\alpha^1, \cdots, \alpha^l$.

We now summarize our discussion as the following result. For convenience we set $T=1$ here. (Thus, for homogeneous $f$, we don't need to couple $tT^2=1$ to get a local index theorem.)

\begin{thm} For the Landau-Ginzburg model $(\mathbb C^n, g_0, f)$ where $f$ is a nondegenerate homogeneous polynomial with weight $q$, we have the following small time asymptotic expansion of the heat trace for the Witten Laplacian:
	\[ \operatorname{Tr}\left(\exp (-t\Box_{f})\right) \sim \frac{1}{(4\pi t)^{n}} 
	\sum_{j=0}^{\infty} \sum_{l \leq j}\sum_{\alpha^1, \cdots, \alpha^l } t^{j+ \delta q \sum_{i=1}^l |\alpha^i| -\delta l - 2n\delta q} C_{\alpha^1, \cdots, \alpha^l}(f),  \]
	as $t\rightarrow 0$, where $|\alpha^1| + \cdots |\alpha^l| \leq 2j$. Moreover, for $k$ sufficiently large, and $t\in (0, 1]$, 
	\begin{align*}
	\left| \operatorname{Tr}\left(\exp (-t\Box_{f})\right)- \frac{1}{(4\pi t)^{n}} 
	\sum_{j=0}^{k}  \sum_{l \leq j}\sum_{\alpha^1, \cdots, \alpha^l } t^{j+ \delta q \sum_{i=1}^l |\alpha^i| -\delta l - 2n\delta q} C_{\alpha^1, \cdots, \alpha^l}(f) \right|  
	\leq Ct^{\frac{k+1}{3} -n -2n \delta q }. 
	\end{align*} 
	Here $\delta=\frac{1}{2(1-q)}$.
\end{thm}

\begin{proof}
	We note that $\k=0$ in this case. The result follows from combining the above discussion.
\end{proof}

\begin{rem}
	A similar but different expansion is in \cite{fanfang2016torsion}, and without the remainder estimate.
\end{rem}

\def\epnh{\ep^{-\frac{1}{2}}}
\def\eph{\ep^{\frac{1}{2}}}
\def\i{\sqrt{-1}}
\def\q{(-1)^{\frac{1}{4}}}
\def\Tr{\operatorname{Tr}}

\bibliography{lib}

\bibliographystyle{plain}
\end{document}